\documentclass{article}

\usepackage{amsmath,amsfonts,amssymb}

\usepackage{graphicx}

\usepackage{dsfont}
\usepackage{bbold}
\usepackage{amsthm}
\usepackage[utf8]{inputenc}
\usepackage{appendix}
\usepackage[colorlinks, allcolors=blue]{hyperref}

\newtheorem{thm}{Theorem}

\newtheorem{prp}{Proposition}
\newtheorem{claim}{Claim}

\title{\textbf{Maximal operators on hyperbolic triangles}}

\author{Romain Branchereau \& Samuel Bronstein \& Anthony Gauvan}

\begin{document}

\maketitle

\begin{abstract}
We characterize the boundedness properties on the spaces $L^p( \mathbb{H}^2)$ of the maximal operator $M_\mathcal{B}$ where $\mathcal{B}$ is an arbitrary family of hyperbolic triangles stable by isometries.
\end{abstract}

\section{Introduction}

The theory of maximal operators have a long history in analysis and those objects are still intensively studied today: in a measured metric space $(X,\mu,d)$, it is common to consider the \textit{centered} \textit{Hardy-Littlewood} maximal operator $M^c$ defined as $$M^c f(x) := \sup_{r > 0}  \frac{1}{\mu\left( B(x,r) \right)} \int_{B(x,r)} |f|$$ or its non \textit{non-centered} version $M$ defined as $$Mf(x) := \sup_{x \in B \in \mathcal{Q}}  \frac{1}{\mu\left( B \right)} \int_{B} |f|.$$ Here, we have denoted by $\mathcal{Q}$ the family containing all the balls of the space $X$.

In the Euclidean space $\mathbb{R}^n$ endowed with the Lebesgue measure, the operator $M$ and $M^c$ are pointwise comparable (\textit{i.e.} for any function $f$ and $x \in \mathbb{R}^n$, we have $ M^c f(x) \simeq_n Mf(x)$) and both operators have weak-type $(1,1)$. 

\begin{thm}[Maximal Theorem in $\mathbb{R}^n$]\label{Maximal}
For any $f$ and $t > 0$, one has the following estimate $$\left| \left\{ Mf > t \right\} \right|_n \lesssim_n \int_{\mathbb{R}^n}\frac{|f|}{t}.$$
\end{thm}

It is well known that the proof of Theorem \ref{Maximal} relies on the classic Vitali's covering Theorem. Also by interpolation arguments with $p = \infty$, we immediately obtain that the operators $M$ and $M^c$ are bounded from $L^p$ to $L^p$ for any $p > 1$. Such an estimate coupled with classic arguments of measure theory yield the following Theorem of differentiation.

\begin{thm}[Lebesgue's Theorem of differentiation]
For any locally integrable function $f$, we have almost everywhere the following identity $$ f(x) = \lim_{r \rightarrow 0} A_r f(x).$$ Here, $A_r f(x)$ stands for the average of $f$ on a ball (or a cube) centered at the point $x$ or radius $r$.
\end{thm}

If the centered or non centered version maximal operator behave in the same fashion in the Euclidean space the story is quite different on non-compact symmetric spaces. Precisely, let $\mathbb{G}$ be a non-compact connected semisimple Lie group with finite center, $\mathbb{K}$ a maximal compact subgroup and consider the non-compact symmetric space $$\mathbb{X} = \mathbb{G}/\mathbb{K}$$ equipped with its natural $\mathbb{G}$-invariant measure $\mu$ and distance $d$.
The case of the centered Hardy-Littlewood operator $M^c$ in this setting has been studied in \cite{St} by Stromberg.

\begin{thm}[Stromberg]\label{TStrom}
On $\left( \mathbb{X},\mu, d\right)$, the centered Hardy-Littlewood operator $M^c$ has weak-type $(1,1)$.
\end{thm}

In contrast with the Euclidean space, one cannot use covering argument on non-compact symmetric space since the growth of the volume of the balls are exponential. To deal with this difficulty, Stromberg uses a local-global strategy to decompose the operator $M^c$ and to detail its optimal boundedness property. On the other hand, the non-centered operator $M$ has been studied by Ionescu in \cite{IO}.

\begin{thm}[Ionescu]\label{TIo}
On $\left( \mathbb{X},\mu, d\right)$, the non-centered Hardy-Littlewood operator $M$ is bounded from $L^p$ to $L^p$ in the sharp range of exponents $p \in (2,\infty]$.
\end{thm}

The proof of Ionescu relies on the relation between the boundedness property of maximal operator and the covering property of the geometric family that defines it. Precisely, in a measured metric space $(X,\mu,d)$, given an arbitrary family $\mathcal{B}$ of measurable sets of finite measure (the \textit{geometric data}) one can define the maximal operator $M_\mathcal{B}$ associated as $$M_\mathcal{B}f(x) := \sup_{x \in R \in \mathcal{B}}  \frac{1}{\mu\left( R \right)} \int_R |f|.$$ Here, we have implicitly supposed that we have $$X = \bigcup_{R \in \mathcal{B}} R.$$ In \cite{CF}, Cordoba and Fefferman say that such a family $\mathcal{B}$ satisfies the covering property $V_q$ for $q \in (1,\infty)$ if given any finite family $\left\{ R_i : i \in I \right\}$ included in $\mathcal{B}$, there exists a subfamily $J \subset I$ such that we have $$ \mu\left( \bigcup_{i \in I} R_i \right) \lesssim \mu\left( \bigcup_{j \in J} R_j \right)$$ and also $$\left\|\sum_{j \in J} \mathbb{1}_{R_j} \right\|_{L^p(\mu)} \lesssim \mu\left( \bigcup_{j \in J} R_j \right)^\frac{1}{q}.$$ They proved then the following.

\begin{thm}[Cordoba and R. Fefferman]
The family $\mathcal{B}$ satisfies the covering property $V_q$ if and only if the associated maximal operator $M_\mathcal{B}$ has weak-type $(p,p)$ \textit{i.e.} if for any $f$ and $t>0$, one has $$ \mu\left( \left\{ M_\mathcal{B}f > t \right\} \right)^{\frac{1}{p}} \lesssim \frac{\|f\|_{L^p(\mu)}}{t}.$$ Here we have supposed that $\frac{1}{p} + \frac{1}{q}  = 1$. 
\end{thm}

Hence, in regards of Theorems \ref{TStrom} and \ref{TIo}, one can see that the \textit{geometry} deeply influence the behavior of $M^c$ or $M$. Let us consider a slightly different problematic: in a fixed measured metric space $(X,\mu,d)$ and with this general definition of a maximal operator $M_\mathcal{B}$, one can wonder what happens when the family $\mathcal{B}$ is composed of more complex geometric objects than balls ? With appropriate geometric restrictions on the family $\mathcal{B}$, this problematic can unravel deep phenomena of the underlying space $X$ and the interaction of the elements of the family $\mathcal{B}$.

Let us give an example in the Euclidean plane $\mathbb{R}^2$: given an arbitrary set of directions $\Omega \subset \mathbb{S}^1$, consider the \textit{directional} family $\mathcal{R}_\Omega$ composed of every rectangles whose longest side makes an angle $\omega \in \Omega$ with the $Ox$-axis and denote by $M_\Omega$ the maximal operator associated to this family. Those \textit{directional} maximal operators have been intensively studied in the past decades (see for examples \cite{ALFONSECA}, \cite{BATEMANKATZ} and \cite{NSW}) and it was only after a long serie of different works that Bateman classified in \cite{BATEMAN} the behavior of directional maximal operators on the $L^p$ spaces for $1 < p < \infty$. We give some examples before stating Bateman's Theorem: \begin{itemize}
    \item if the set of directions $\Omega$ is finite, then the maximal operator $M_\Omega$ is bounded on $L^p$ for any $p > 1$.
    \item if we consider the set of directions $\Omega_{\text{lac}} = \left\{ \frac{\pi}{2^k} : k \geq 1 \right\}$, then the maximal operator $M_{\Omega_{\text{lac}}}$ is bounded on $L^p$ for any $p > 1$, see \cite{NSW}. The proof relies on Fourier analysis techniques
    \item if we consider the set of directions $\Omega_{\text{lin}} = \left\{ \frac{\pi}{k} : k \geq 1 \right\}$, then the maximal operator $M_{\Omega_{\text{lin}}}$ is not bounded on $L^p$ for any $p < \infty$. The proof relies on the construction of a so-called \textit{Perron's tree} which is a concrete realization of a \textit{Kakeya blow} in the plane.
    \item if we consider a classic ternary Cantor set $\Omega_{\text{Cantor}}$ embedded into $\mathbb{S}^1$, then the maximal operator $M_{\Omega_{\text{Cantor}}}$ is not bounded on $L^p$ for any $p < \infty$, see \cite{BATEMANKATZ}. As for $\Omega_{\text{lin}}$, the proof relies on the realization of a Kakeya blow in the plane but \textit{via} a sophisticated random construction.
\end{itemize} Here, we have used the term Kakeya blow to designate the following \textit{geometric obstruction} of the Euclidean space (we denote by $\mathcal{R}$ the family containing every plane rectangle in $\mathbb{R}^2$).

\begin{thm}[Kakeya blow with $\mathcal{R}$]\label{BLOW}
Given any large constant $A \gg 1$, there exists a finite family of rectangles $$\left\{ R_i : i \in I \right\} \subset {\mathcal{R}}$$ such that we have $$\left| \bigcup_{i \in I} TR_i \right| \geq A\left| \bigcup_{i \in I} R_i \right|.$$ Here, we have denoted by $TR$ the rectangle $R$ translated along its longest side by its own length
\end{thm}

This Theorem has several applications in analysis and is related to the traditional Kakeya problem. For example, C. Fefferman used a more refined version of this property in order to disprove the famous Ball multiplier Theorem: see \cite{F} which also contains a proof of Theorem \ref{BLOW}. For an arbitrary set of directions $\Omega \subset \mathbb{S}^1$, let us state Bateman's Theorem which indicates that the correct notion to consider is \textit{finiteness lacunarity}. We invite the reader to look at Bateman's paper \cite{BATEMAN} for a precise definition of this notion.

\begin{thm}[Bateman]\label{T0}
We have the following alternative:
\begin{itemize}
    \item if the set of directions $\Omega$ is finitely lacunary then $M_\Omega$ is bounded on $L^p$ for any $p > 1$.
    \item if the set of directions $\Omega$ is not finitely lacunary then it is possible to make a Kakeya blow with the family $\mathcal{R}_\Omega$ and in particular, $M_\Omega$ is not bounded on $L^p$ for any $p < \infty$.
\end{itemize}
\end{thm}

In this text, we investigate this problematic in the hyperbolic plane $\mathbb{H}^2$ endowed with its natural measure $\mu$ and where $\mathcal{B}$ is a family included in $\mathcal{T}$ - the family containing all the geodesic triangles of $\mathbb{H}^2$ - and where in addition $\mathcal{B}$ is stable by isometries of $\mathbb{H}^2$. Since we are able to completely determine the behavior of the operator $M_\mathcal{B}$, it seems interesting to study less regular geometric families and the following example seems pertinent. We fix a point $\omega \in \partial \mathbb{H}^2$ and we consider the family of curves $\mathcal{G}_\omega$ which contains every geodesics that goes to $\omega$ at infinite and also the family $\mathcal{H}_\omega$ which contains every horocycles tangent to $\omega$: given two  geodesics $d \neq d' \in \mathcal{G}_\omega$ and two horocycles $h \neq h' \in \mathcal{H}_\omega$, we consider the open bounded set $P(d,d',h,h')$ whose border is delimited by the four curves $d,d',h$ and $h'$ and then we define the family $$ \mathcal{P}_\omega = \left\{P(d,d',h,h') : d \neq d' \in \mathcal{G}_\omega, h \neq h' \in \mathcal{H}_\omega  \right\}.$$ What can be said about the maximal operator $$M_{\mathcal{P}_\omega} : L^\infty \rightarrow L^\infty $$ associated to the family $\mathcal{P}_\omega$ ? Can we expect boundedness property analogue to the strong maximal operator ? Thereafter, given an arbitrary family $\Omega \subset \partial \mathbb{H}^2$, one can consider the family $$ \mathcal{P}_\Omega = \bigcup_{\omega \in \Omega} \mathcal{P}_\omega.$$ What can be said about the operator $M_{\mathcal{P}_\Omega}$ associated to this family ? Does its behavior depend on $\Omega$ in the same fashion that in the Euclidean case ? We will return to those questions later.

\section{Results}

As said earlier, we are interested by families composed of geodesic triangles which are invariant under the action of isometries. In the following, we denote by by $G$ the group of isometries of $\mathbb{H}^2$ and by $\mathcal{T}$ the family containing all geodesic triangles in $\mathbb{H}^2$. A family $\mathcal{B}$ included in $\mathcal{T}$ and invariant by $G$ is parameterized by the angles exhibited by the triangles it contains: if we denote by $\boldsymbol{\alpha}_T$ the set of angles of the triangle $T$ and that we let $A_\mathcal{B} :=  \{ \boldsymbol{\alpha}_T : T \in \mathcal{B} \}$ then, because $\mathcal{B}$ is invariant under $G$, we precisely have $$\mathcal{B} =  \{ T \in \mathcal{T} : \boldsymbol{\alpha}_T \in A_\mathcal{B} \}.$$ We consider the simplex $S$ of admissible angles which is defined as $$S := \left\{ \boldsymbol{\alpha} \in \left(\mathbb{R}_+\right)^3 :  s(\boldsymbol{\alpha}) \leq \pi \right\}$$ where $s(\boldsymbol{\alpha}) := \alpha_1 + \alpha_2 + \alpha_3$. We are going to decompose the simplex $S$ in three disjoint parts $S_1, S_2$ and $S_3$ according to the \textit{eccentricity} function $p$ defined as $$p(\boldsymbol{\alpha}) := \alpha_1\alpha_2\alpha_3.$$ Precisely we define:
\begin{itemize}
    \item $S_0 := \{ \boldsymbol{\alpha} \in S : s(\boldsymbol{\alpha}) < \pi, p(\boldsymbol{\alpha}) > 0 \}$
    \item $S_1 := \{ \boldsymbol{\alpha} \in S : s(\boldsymbol{\alpha}) = \pi, p(\boldsymbol{\alpha}) > 0 \}$,
    \item $S_2 := \{ \boldsymbol{\alpha} \in S : s(\boldsymbol{\alpha}) < \pi, p(\boldsymbol{\alpha}) = 0\}$,
    \item and $S_3 := \{ \boldsymbol{\alpha} \in S : s(\boldsymbol{\alpha})= \pi, p(\boldsymbol{\alpha}) = 0\}$
\end{itemize}
The functions $s$ and $p$ have simple geometric interpretations: on one hand for any triangle $T$, one has $$\mu\left(T\right) = \pi - s(\boldsymbol{\alpha}_T)$$ and on the other hand, $p(\boldsymbol{\alpha}_T)$ represents the eccentricity of the triangle $T$, its \textit{thickness}. We can now state our main result.

\begin{thm}\label{T:MAIN}
Let $\mathcal{B}$ be a family of hyperbolic triangles stable by isometries and recall its set of angle is defined as $A := \{ \boldsymbol{\alpha}_T : T \in \mathcal{B} \}$. We have the following alternatives: \begin{itemize}
    \item (I) if $\Bar{A} \subset S_0$, then the operator $M_\mathcal{B}$ is bounded from $L^1$ to $L^1$ \textit{i.e.} there exists a constant $C_\mathcal{B} < \infty$ such that for any function $f$ we have $$\|M_\mathcal{B}f \| \leq C_\mathcal{B}\|f \|_{L^1(\mu)}.$$
    \item (II) if $\Bar{A} \subset S_1$ and $\Bar{A} \cap S_1 \neq  \emptyset$, then the operator $M_\mathcal{B}$ has weak-type $(1,1)$ \textit{i.e.} there exists a constant $C_\mathcal{B} < \infty$ such that for any function $f$ and $t > 0$, the following estimate holds $$\mu\left( \left\{M_\mathcal{B}f > t \right\} \right) \leq C_\mathcal{B}\frac{\|f \|_{L^1(\mu)}}{t}.$$
    \item (III) if $\Bar{A} \cap S_2 \neq  \emptyset$, then for any non zero function $f$, there exists a positive constant $c > 0$ depending on $f$ and $\mathcal{B}$ such that for any $y \in \mathbb{H}^2$, one has $$M_\mathcal{B}f(y) > c.$$
    \item (IV) if $\Bar{A} \cap S_3 \neq  \emptyset$, then for any large constant $C \gg 1$, there exists a bounded set $E$ in $\mathbb{H}^2$ satistying $$ \mu\left( \left\{ M_\mathcal{B}\mathbb{1}_E \geq \frac{1}{4} \right\} \right) \geq C\mu\left(E \right).$$
\end{itemize}
\end{thm}

The proof of Theorem \ref{T:MAIN} will be done in different parts and relies on different ideas: in contrast with the work of Stromberg \cite{St}, we exclusively use geometric techniques. To begin with, if $\Bar{A} \subset S_0$, then the operator $M_\mathcal{B}$ is controlled by $$T :  f(x) \rightarrow  \frac{1}{B(x,r)} \int_{B(x,r)} |f|$$ where the radius $r$ is independent of $x$. Such an operator $T$ is bounded from $L^1$ to $L^1$ and since $M_\mathcal{B} \leq T$, the conclusion comes easily. If $\Bar{A} \subset S_1$ and $\Bar{A} \cap S_1 \neq  \emptyset$, this means that the triangles in the family $\mathcal{B}$ cannot become arbitrarily thin. Hence, one can compare each of them to ball whose radius is uniformly bounded and this will allow us to prove the desired weak-type estimate thanks to a covering argument \textit{à la} Vitali.

In the case where $\Bar{A} \cap S_2 \neq  \emptyset$, the family $\mathcal{B}$ contains triangles that tend to become partly ideal. We prove that given any pair of point $(x_0,y)$ one can always catch the point $y$ and a small ball centered at $x_0$ with partly ideal triangles and this will allow us to prove that the maximal function $M_\mathcal{B}f$ of any non zero function is uniformly bounded by below. This effect is specific to the hyperbolic plane $\mathbb{H}^2$ and cannot be observed in the Euclidean space with classic convex sets.

Finally, if $\Bar{A} \cap S_3 \neq  \emptyset$ then the family $\mathcal{B}$ contains arbitrarily thin triangles that tend to become Euclidean. We will be able to construct the desired set $E$ by considering local families of triangles: we will prove that their images by the exponential map contain enough Euclidean triangles in an arbitrary tangent space $T_{x_0}\mathbb{H}^2$ which will allow us to exploit a Kakeya-type set of the Euclidean plane.

\section{Thick triangles (I)}

We suppose that we have $\Bar{A} \subset S_0$. For any $x \in \mathbb{H}^2$, consider the family $\mathcal{B}(x)$ defined as $$ \mathcal{B}(x) := \left\{ 
 T \in \mathcal{B} : x \in T\right\}.$$ In other words, $\mathcal{B}(x)$ is composed of every triangles $T$ in $\mathcal{B}$ which contains the point $x$. We claim the following.
\begin{claim}
If $\Bar{A} \subset S_0$ then there exists $r = r(\mathcal{B}) >0$ such that for any $T \in \mathcal{B}(x)$, one has $T \subset B(x,r)$ and also $\mu\left( T \right) \simeq_\mathcal{B} \mu\left( B(x,r) \right).$
\end{claim}

With this proposition at hands, the conclusion comes easily: for any $f$, we have $$ \int_{\mathbb{H}^2} M_\mathcal{B}f(x)d\mu(x) \lesssim_\mathcal{B} \int_{\mathbb{H}^2 \times \mathbb{H}^2} |f|(y) \mathbb{1}_{ d(x,y) \leq r }d\mu(y) d\mu(x).$$ Indeed for any $x \in \mathbb{H}^2$, there exists $T \in \mathcal{B}(x)$ such that $$M_\mathcal{B}f(x) \leq \frac{2}{\mu\left( T \right)} \int_T |f|  \lesssim_\mathcal{B} \int_{B(x,r)} |f|$$ since we have $T \subset B(x,r)$ and that both sets have comparable volume. Using Fubini we obtain as expected $$ \int_{\mathbb{H}^2} M_\mathcal{B}f(x)d\mu(x) \lesssim_\mathcal{B} \|f\|_{L^1(\mu)}.$$

\section{Thick triangles (II)}

We suppose that the family $\mathcal{B}$ satisfies $\Bar{A} \subset S_1$ and $\Bar{A} \cap S_1 \neq  \emptyset$. This condition means that the eccentricity of a triangle in $\mathcal{B}$ is uniformly bounded by below but we do not have control on its volume now. The following Claim holds.

\begin{claim}\label{CLAIMX}
For any triangle $T \in \mathcal{B}$, there exists a ball $B(T)$ containing $T$ and satisfying $$\mu\left( B(T)  \right) \simeq_\mathcal{B} \mu\left( T  \right).$$
\end{claim}

In particular, there exists a constant $R < \infty $ only depending on $\mathcal{B}$ such that $$\sup_{T \in \mathcal{B}} \text{diam}(B(T)) < R.$$ We recall now the Vitali covering Theorem which is valid in metric space.

\begin{thm}[Vitali covering Theorem]\label{THMVIT}
Given any family $\mathcal{F}$ of balls whose radius is uniformly bounded by a constant $R < \infty$, there exists a sub-family $\mathcal{F}' \subset \mathcal{F}$ satisfying the following properties: on one hand the balls of $\mathcal{F}'$ are pairwise disjoint and on the other hand, one has $$\bigcup_{B \in \mathcal{F}}B \subset   \bigcup_{B' \in \mathcal{F}'}5B'.$$ Here $5B$ stands for the ball $B(x,5r)$ if $B=B(x,r)$.
\end{thm}

With Claim \ref{CLAIMX} and Theorem \ref{THMVIT} at hands, we can prove that $M_\mathcal{B}$ has weak-type $(1,1)$. We fix $f$ a function on $\mathbb{H}^2$ and $t > 0$: by definition the level set $\left\{ M_\mathcal{B}f > t \right\}$ is the union of the triangles $T$ belonging to $\mathcal{B}$ satisfying $$\int_T |f| > \mu(T)t.$$ We denote by $\mathcal{F}$ this family of triangles and we apply Claim \ref{CLAIMX} for each $T$ in $\mathcal{F}$ and then we apply Theorem \ref{THMVIT} to each balls $B(T)$. This gives us a subfamily $\mathcal{F}'$ of $\mathcal{F}$ satisfying the following inclusion $$  \left\{ M_\mathcal{B}f > t \right\} \subset \bigcup_{T' \in \mathcal{F}'} 5B(T')$$ and moreover, the balls $B(T')$ for $T' \in \mathcal{F}'$ are pairwise disjoint. Since all the balls considered have a radius uniformly bounded, we have $$\mu\left(5B(T')\right) \simeq_\mathcal{B} \mu\left(B(T')\right)$$ and so the following estimate holds $$\mu\left( \bigcup_{T' \in \mathcal{F}'} 5B(T') \right) \leq \sum_{T' \in \mathcal{F}'} \mu\left(5B(T')\right) \lesssim_{\mathcal{B}} \sum_{T' \in \mathcal{F}'}\mu\left(B(T')\right).$$ Now we also have $\mu\left(B(T') \right) \simeq_\mathcal{B}  \mu\left(T'\right)$ and so we obtain $$\sum_{T' \in \mathcal{F}'} \mu\left(B(T')\right) \lesssim_\mathcal{B} \sum_{T' \in \mathcal{F}'} \frac{1}{t} \int_{B(T')} |f|.$$ Since the balls $B'$ are pairwise disjoint we obtain as expected $$\mu\left(\left\{M_\mathcal{B}f > t \right\}\right) \lesssim_\mathcal{B} \frac{1}{t}\int |f|.$$

\section{Partly ideal triangles}

We suppose now that the family $\mathcal{B}$ satisfies $\Bar{A} \cap S_2 \neq \emptyset$. We fix an element $\boldsymbol{\alpha} \in \Bar{A} \cap S_2$ and without loss of generality we suppose that it is of the form $\boldsymbol{\alpha} = (0, \alpha_2, \alpha_3)$. We prove the following Proposition.

\begin{prp}\label{THMCATCH}
For any $x_0 \in \mathbb{H}^2$, there exists $r>0$ such that for any $y\in\mathbb{H}^2$, there is a triangle $T$ satisfying $\boldsymbol{\alpha}_T = (0, \alpha_2, \alpha_3)$, $y \in T $ and $B(x_0,r) \subset T$.
\end{prp}

%%%%%%%%%%%%%%%%%%%%%%%%%%%%%%%%%%%%%%%%%%%%%%%%%%%%%%%%%%%%%%%%%%%%%%%%%%%%%%%%%%%%
%%%%%%%%%%%%%%%%%%%%%%%%%%%%%%%%%%%%%%%%%%%%%%%%%%%%%%%%%%%%%%%%%%%%%%%%%%%%%%%%%%%%
%%%%%%%%%%%%%%%%%%%%%%%%%%%%%%%%%%%%%%%%%%%%%%%%%%%%%%%%%%%%%%%%%%%%%%%%%%%%%%%%%%%%
%%%%%%%%%%%%%%%%%%%%%%%%%%%%%%%%%%%%%%%%%%%%%%%%%%%%%%%%%%%%%%%%%%%%%%%%%%%%%%%%%%%%

\begin{figure}[h!]
\centering
\includegraphics[scale=0.8]{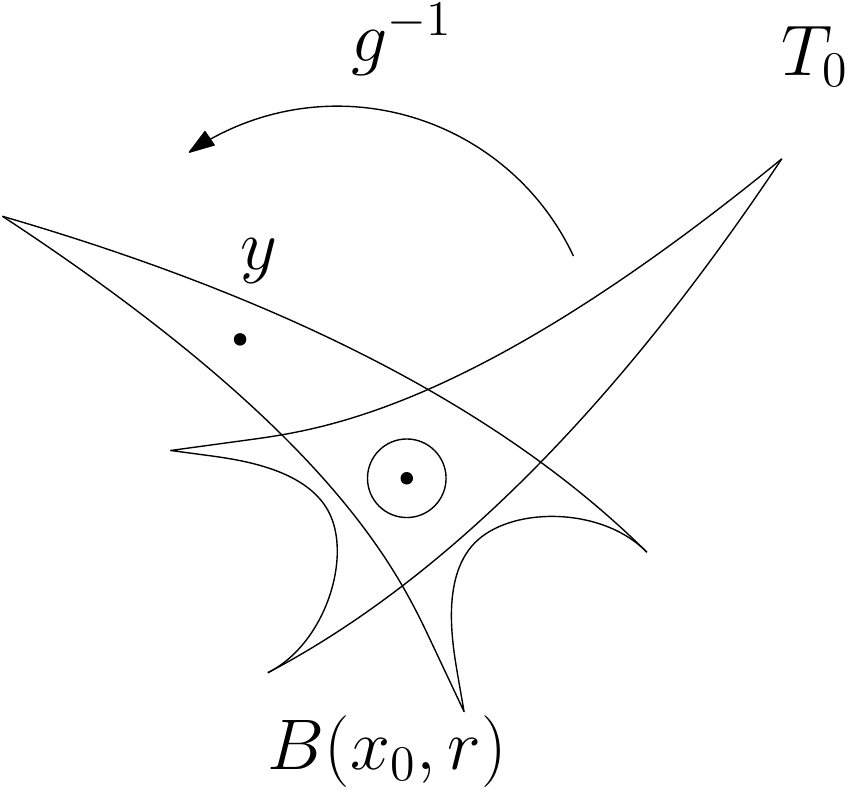}
\caption{An illustration of the proof of Theorem \ref{THMCATCH}.}
\end{figure}

%%%%%%%%%%%%%%%%%%%%%%%%%%%%%%%%%%%%%%%%%%%%%%%%%%%%%%%%%%%%%%%%%%%%%%%%%%%%%%%%%%%%
%%%%%%%%%%%%%%%%%%%%%%%%%%%%%%%%%%%%%%%%%%%%%%%%%%%%%%%%%%%%%%%%%%%%%%%%%%%%%%%%%%%%
%%%%%%%%%%%%%%%%%%%%%%%%%%%%%%%%%%%%%%%%%%%%%%%%%%%%%%%%%%%%%%%%%%%%%%%%%%%%%%%%%%%%
%%%%%%%%%%%%%%%%%%%%%%%%%%%%%%%%%%%%%%%%%%%%%%%%%%%%%%%%%%%%%%%%%%%%%%%%%%%%%%%%%%%%

%%%%%%%%%%%%%%%%%%%%%%%%%%%%%%%%%%%%%%%%%%%%%%%%%%%%%%%%%%%%%%%%%%%%%%%%%%%%%%%%%%%%
%%%%%%%%%%%%%%%%%%%%%%%%%%%%%%%%%%%%%%%%%%%%%%%%%%%%%%%%%%%%%%%%%%%%%%%%%%%%%%%%%%%%
%%%%%%%%%%%%%%%%%%%%%%%%%%%%%%%%%%%%%%%%%%%%%%%%%%%%%%%%%%%%%%%%%%%%%%%%%%%%%%%%%%%%
%%%%%%%%%%%%%%%%%%%%%%%%%%%%%%%%%%%%%%%%%%%%%%%%%%%%%%%%%%%%%%%%%%%%%%%%%%%%%%%%%%%%

\begin{figure}[h!]
\centering
\includegraphics[scale=0.8]{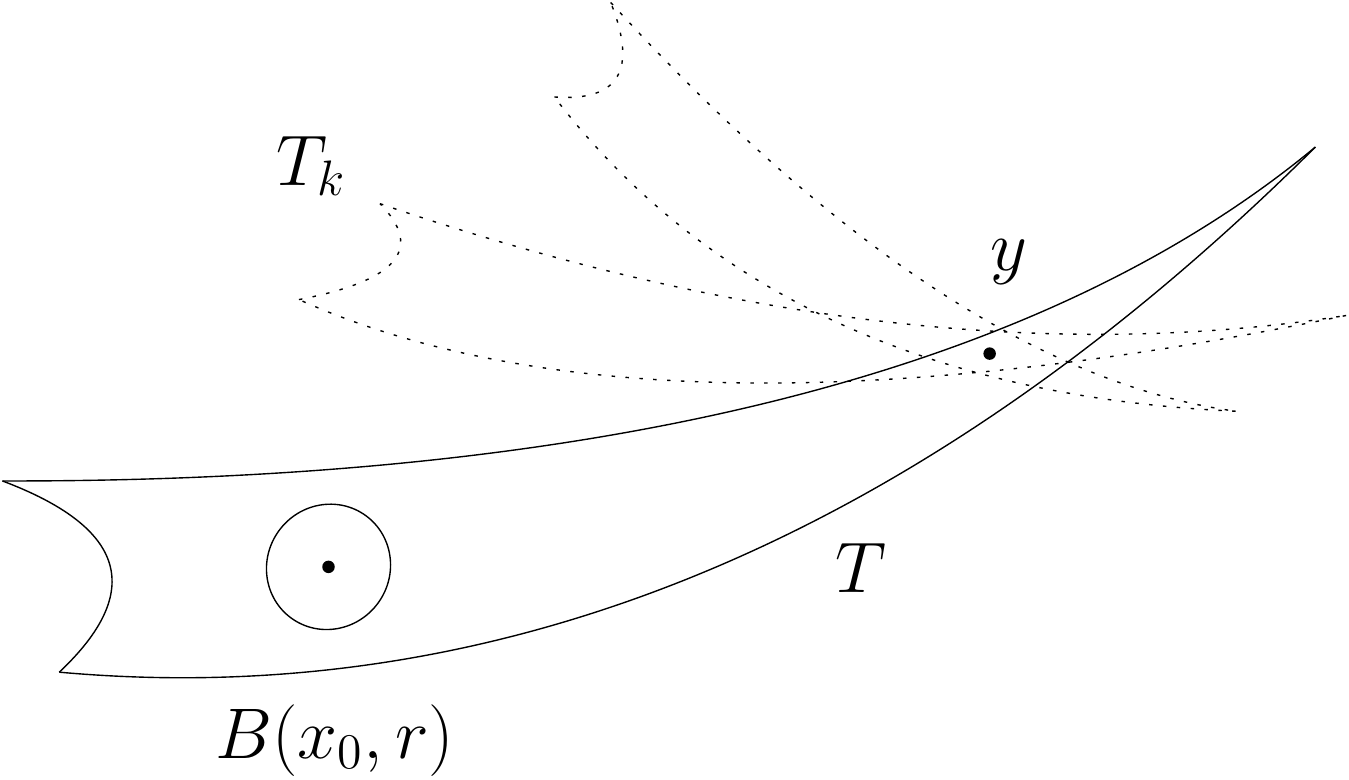}
\caption{The existence of the sequence $\{ T_k : k \geq 1 \}$ comes from the fact that $\boldsymbol{\alpha}$ belongs to $\Bar{A}$.}
\end{figure}

%%%%%%%%%%%%%%%%%%%%%%%%%%%%%%%%%%%%%%%%%%%%%%%%%%%%%%%%%%%%%%%%%%%%%%%%%%%%%%%%%%%%
%%%%%%%%%%%%%%%%%%%%%%%%%%%%%%%%%%%%%%%%%%%%%%%%%%%%%%%%%%%%%%%%%%%%%%%%%%%%%%%%%%%%
%%%%%%%%%%%%%%%%%%%%%%%%%%%%%%%%%%%%%%%%%%%%%%%%%%%%%%%%%%%%%%%%%%%%%%%%%%%%%%%%%%%%
%%%%%%%%%%%%%%%%%%%%%%%%%%%%%%%%%%%%%%%%%%%%%%%%%%%%%%%%%%%%%%%%%%%%%%%%%%%%%%%%%%%%

\begin{proof}
Let $x_0 \in \mathbb{H}^2$ and consider a triangle $T_0$ of angles $(0, \alpha_2, \alpha_3)$. As the action of $G$ on $\mathbb{H}^2$ is transitive, we may assume that $T_0$ contains $x_0$ in its interior: fix then  $r>0$ such that $T_0$ contains the ball $B(x_0,r)$.

Given any point $y\in\mathbb{H}^2$, its orbit under the stabilizer of $x_0$ is exactly the circle $C(x_0,d(x_0,y))$. As $T_0$ has at least one angle equal to $0$ by hypothesis, it is noncompact and intersect nontrivially $C(x_0,d(x_0,y))$. Hence there is $g\in G$ such that $g(x_0)=x_0$ and $g(y)\in T_0$ \textit{i.e.} $y \in T:=g^{-1}T_0$ and we have $$B(x_0,r) \subset T.$$
\end{proof}

Proposition \ref{THMCATCH} yields the conclusion as follow. Consider a non zero function $f$ on $\mathbb{H}^2$:  there exists a ball $B(x_0,r)$ such that $\int_{B(x_0,r)}|f|>0$ and by a standard compactness argument, we can choose $r$ to be arbitrarily small. We fix any point $y$ in $\mathbb{H}^2$ and we apply Theorem \ref{THMCATCH} which gives us a triangle $T$ satisfying $\boldsymbol{\alpha}_T = (0, \alpha_2, \alpha_3)$, $y \in T$ and $B(x_0,r) \subset T$. To conclude, it suffices to observe that since $\boldsymbol{\alpha} \in \Bar{A} \cap S_2$, there exists a sequence of triangles $\left\{ T_k : k \geq 1 \right\} \subset \mathcal{B}$ such that each of them contain $y$ and satistying $$\int_{T_k} |f| \rightarrow \int_{T} |f|.$$ In particular we have $$M_\mathcal{B}f(y) \geq \frac{1}{\mu(T)}\int_T |f| \geq \frac{1}{\pi} \int_{B(x_0,r)} |f|.$$ This yields the desired conclusion since the integral $\int_{B(x_0,r)} |f|$ is independant of the point $y$ (it only depends on $f$ and $\mathcal{B}$).

\section{Thin and small triangles}

We suppose now that the family $\mathcal{B}$ satisfies $\Bar{A} \cap S_3 \neq \emptyset$. We are going to prove that for any constant $C \gg 1$, there exists a bounded set $E$ in $\mathbb{H}^2$ satisfying $$ \mu\left( \left\{ M_\mathcal{B}\mathbb{1}_E \geq \frac{1}{4} \right\} \right) \geq C\mu\left(E \right).$$ In particular, this property implies that the operator $M_\mathcal{B}$ is not bounded on the space $L^p(\mathbb{H}^2)$ for any $p < \infty$. To construct the set $E$, the idea is to make a Kakeya blow with the triangles of the family $\mathcal{B}$: this is possible precisely because the family $\mathcal{B}$ contains triangles which are arbitrarily thin and almost Euclidean. To make this argument rigorous, we need to work in a tangent space $T_{x_0}\mathbb{H}^2$ and to precisely describe the image of small triangles in a neighborhood of $x_0$ by the exponential map.

We introduce the following notation: given two triangles $T$ and $T'$ (Euclidean or hyperbolic), we note $$T \sim T'$$ if there is an isometry $g$ such that $T= g(T')$. In the following, we fix any point $x_0 \in \mathbb{H}^2$ and let $\exp_{x_0}$ be the exponential map between $\mathbb{H}^2$ and $T_{x_0}\mathbb{H}^2$. We also suppose that we work in a neighborhood $B(x_0,r)$ of $x_0$ so small that we have the following estimate for any open set $U \subset B(x_0,r)$ and with $h = \frac{1}{1000}$ $$(1-h)\left|\exp_{x_0}(U) \right| \leq \mu\left(U \right) \leq (1+h)\left|\exp_{x_0}(U)\right|.$$ We consider \textit{local} families of triangles: given any $r > 0$ and a triangle $T$, we consider the following family of triangles $$\mathcal{L}_{r}(T) := \left\{ T' \in \mathcal{T} : T' \sim T, T' \subset B(x_0,r) \right\}.$$ Because we want to be able to \textit{move} our triangle $T$ in the ball $B(x_0,r)$, we assume that we have $$\frac{r}{10} > \text{diam}(T).$$ The following proposition states that in the Euclidean plane, a maximal operator defined on a local family of thin triangles can exploit a Kakeya-type set.

%%%%%%%%%%%%%%%%%%%%%%%%%%%%%%%%%%%%%%%%%%%%%%%%%%%%%%%%%%%%%%%%%%%%%%%%%%%%%%%%%%%%
%%%%%%%%%%%%%%%%%%%%%%%%%%%%%%%%%%%%%%%%%%%%%%%%%%%%%%%%%%%%%%%%%%%%%%%%%%%%%%%%%%%%
%%%%%%%%%%%%%%%%%%%%%%%%%%%%%%%%%%%%%%%%%%%%%%%%%%%%%%%%%%%%%%%%%%%%%%%%%%%%%%%%%%%%
%%%%%%%%%%%%%%%%%%%%%%%%%%%%%%%%%%%%%%%%%%%%%%%%%%%%%%%%%%%%%%%%%%%%%%%%%%%%%%%%%%%%

\begin{figure}[h!]\label{IML}
\centering
\includegraphics[scale=0.8]{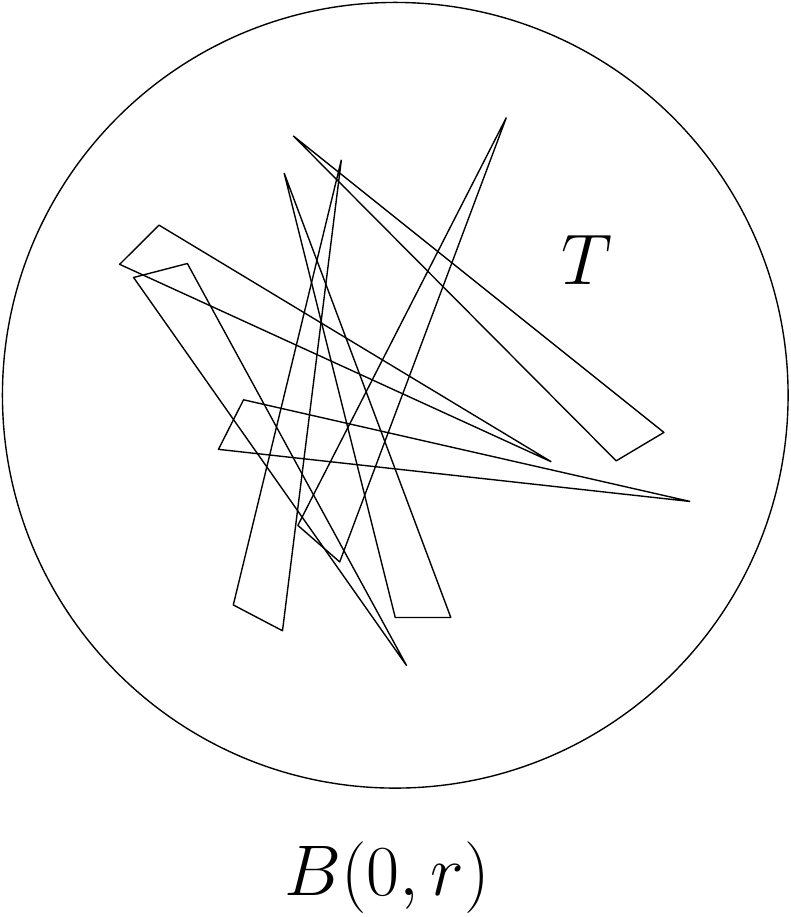}
\caption{Representation of a local family $\mathcal{L}_r(T)$ in the Euclidean plane.}
\end{figure}

%%%%%%%%%%%%%%%%%%%%%%%%%%%%%%%%%%%%%%%%%%%%%%%%%%%%%%%%%%%%%%%%%%%%%%%%%%%%%%%%%%%%
%%%%%%%%%%%%%%%%%%%%%%%%%%%%%%%%%%%%%%%%%%%%%%%%%%%%%%%%%%%%%%%%%%%%%%%%%%%%%%%%%%%%
%%%%%%%%%%%%%%%%%%%%%%%%%%%%%%%%%%%%%%%%%%%%%%%%%%%%%%%%%%%%%%%%%%%%%%%%%%%%%%%%%%%%
%%%%%%%%%%%%%%%%%%%%%%%%%%%%%%%%%%%%%%%%%%%%%%%%%%%%%%%%%%%%%%%%%%%%%%%%%%%%%%%%%%%%

\begin{prp}\label{P1}
For any large constant $C \gg 1$, there exists $0 < \epsilon \ll 1$ such that if $T_i$ is an Euclidean triangle thin enough (\textit{i.e.} satisfying $p(\boldsymbol{\alpha}_{T_i}) < \epsilon$) then any maximal operator $M_{\mathcal{L}_r(T_i)}$ defined on a non empty local family $$\mathcal{L}_r(T_i) := \left\{ T' \in \mathcal{T} : T' \sim T, T' \subset B(0,r) \right\}$$ admits a Kakeya-type set $F$ included in $B(x_0,r)$ satisfying $$ \left| \left\{ M_{\mathcal{L}_r(T_i)}\mathbb{1}_F \geq \frac{1}{4} \right\} \right| \geq C\left|F \right|.$$
\end{prp}

Proposition \ref{P1} is a well-known fact in the literature and we omit its proof. To conclude the proof of Theorem \ref{T:MAIN}, we need the following proposition which simply states that given a local family generated by a small triangle $T$ in $\mathbb{H}^2$, its image under the exponential maps is compatible with a local family generated by a triangle $T_i$ which is close to $T$ according to their sets of angles.

\begin{prp}\label{C8P2}
Fix an arbitrary precision $\delta >0$: there exists $0 < r,\epsilon \ll 1$ (which can be taken arbitrarily small) such that for any set of angles $\boldsymbol{\alpha}_0 \in S$ satisfying $|\pi - s(\boldsymbol{\alpha}_0) | \leq \epsilon$, there exists a set of Euclidean angles $\boldsymbol{\alpha}_1$ satisfying $\|\boldsymbol{\alpha}_0/\boldsymbol{\alpha}_1-1\|_1 \leq \delta$ and such that for any triangle $T$ included in $B(x_0,r)$ and satisfying $$\boldsymbol{\alpha}_T = \boldsymbol{\alpha}_0,$$ there exists an Euclidean triangle $\varphi(T)$ which depends on $T$ satisfying
\begin{flalign}
    \varphi(T) \subset \exp_{x_0}(T)\\
    |\varphi(T)|\geq\frac{1}{10}|T|\\
    \boldsymbol{\alpha}_{\varphi(T)} = \boldsymbol{\alpha}_1.
\end{flalign}
Finally, there exists $0 < r' \ll 1$ such that the following inclusion holds $$\emptyset \neq \mathcal{L}_{r'}(\varphi(T))  \subset \varphi(\mathcal{L}_r(T)).$$
\end{prp}
\begin{proof}
    As it is well known, the differential at $0$ of the exponential map is the identity.
    Hence given $\epsilon>0$, there is $r>0$ sich that the restriction of the exponential map to $B(0,r)$ is
    $(1+\epsilon)$ bilipschitz and quasiconformal.
    In particular, $\exp_{x_0}$ will be uniformly close to an isometry, so the preimage of a hyperbolic triangle
    $T$ in $B(x_0,r)$ with angles $\boldsymbol{\alpha}_0$ will contain a Euclidean triangle with angles $\boldsymbol{\alpha}_1$ satisfying $\|\boldsymbol{\alpha}_1/ \boldsymbol{\alpha}_0-1\|_1 \leq \epsilon$.
    Taking $\epsilon$ small enough will give exactly the claimed properties.
    
\end{proof}

%%%%%%%%%%%%%%%%%%%%%%%%%%%%%%%%%%%%%%%%%%%%%%%%%%%%%%%%%%%%%%%%%%%%%%%%%%%%%%%%%%%%
%%%%%%%%%%%%%%%%%%%%%%%%%%%%%%%%%%%%%%%%%%%%%%%%%%%%%%%%%%%%%%%%%%%%%%%%%%%%%%%%%%%%
%%%%%%%%%%%%%%%%%%%%%%%%%%%%%%%%%%%%%%%%%%%%%%%%%%%%%%%%%%%%%%%%%%%%%%%%%%%%%%%%%%%%
%%%%%%%%%%%%%%%%%%%%%%%%%%%%%%%%%%%%%%%%%%%%%%%%%%%%%%%%%%%%%%%%%%%%%%%%%%%%%%%%%%%%

\begin{figure}[h!]
\centering
\includegraphics[scale=0.8]{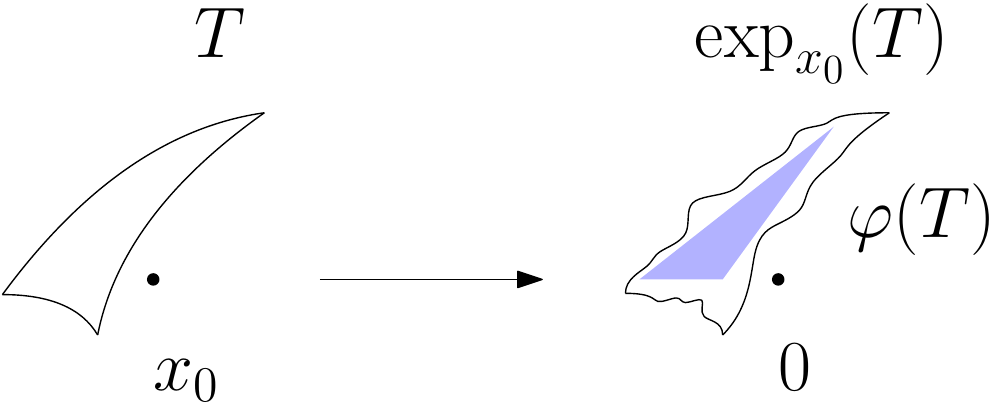}
\caption{Given a triangle $T$ close to $x_0$, we claim that its image $\exp_{x_0}(T)$ contains a triangle $\varphi(T)$ which is not too small compared to $T$ and which is also thin enough.}
\end{figure}

%%%%%%%%%%%%%%%%%%%%%%%%%%%%%%%%%%%%%%%%%%%%%%%%%%%%%%%%%%%%%%%%%%%%%%%%%%%%%%%%%%%%
%%%%%%%%%%%%%%%%%%%%%%%%%%%%%%%%%%%%%%%%%%%%%%%%%%%%%%%%%%%%%%%%%%%%%%%%%%%%%%%%%%%%
%%%%%%%%%%%%%%%%%%%%%%%%%%%%%%%%%%%%%%%%%%%%%%%%%%%%%%%%%%%%%%%%%%%%%%%%%%%%%%%%%%%%
%%%%%%%%%%%%%%%%%%%%%%%%%%%%%%%%%%%%%%%%%%%%%%%%%%%%%%%%%%%%%%%%%%%%%%%%%%%%%%%%%%%%

%%%%%%%%%%%%%%%%%%%%%%%%%%%%%%%%%%%%%%%%%%%%%%%%%%%%%%%%%%%%%%%%%%%%%%%%%%%%%%%%%%%%
%%%%%%%%%%%%%%%%%%%%%%%%%%%%%%%%%%%%%%%%%%%%%%%%%%%%%%%%%%%%%%%%%%%%%%%%%%%%%%%%%%%%
%%%%%%%%%%%%%%%%%%%%%%%%%%%%%%%%%%%%%%%%%%%%%%%%%%%%%%%%%%%%%%%%%%%%%%%%%%%%%%%%%%%%
%%%%%%%%%%%%%%%%%%%%%%%%%%%%%%%%%%%%%%%%%%%%%%%%%%%%%%%%%%%%%%%%%%%%%%%%%%%%%%%%%%%%

\begin{figure}[h!]
\centering
\includegraphics[scale=0.8]{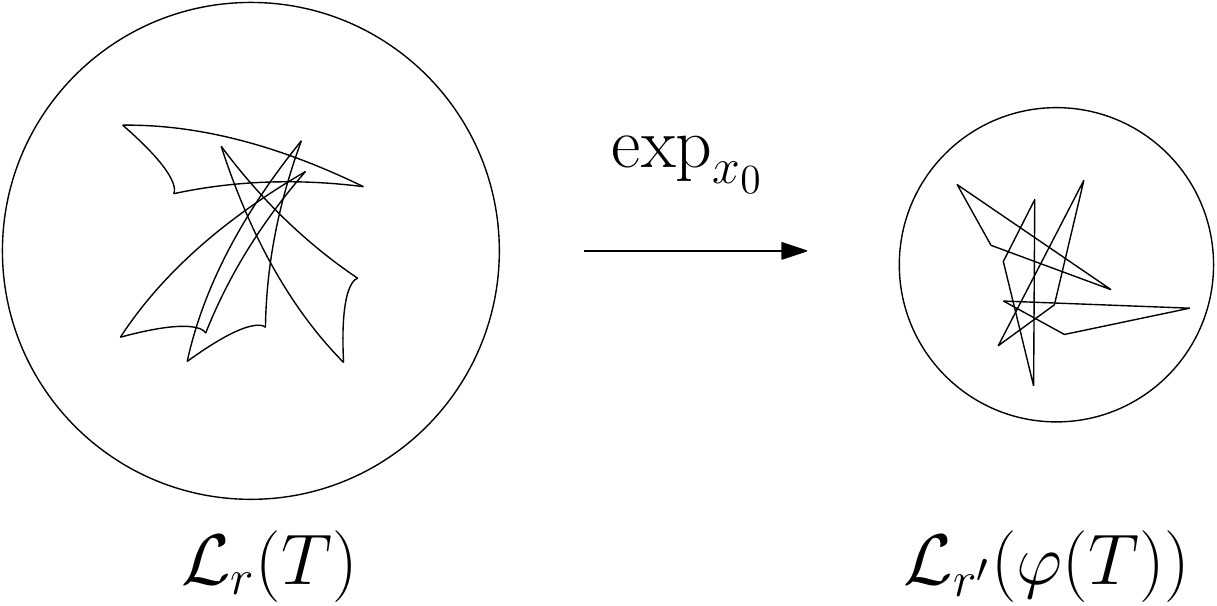}
\caption{We also claim that the image of the local family $\mathcal{L}_r(T)$ under the exponential will contain a local family $\mathcal{L}_{r'}(\varphi(T))$: this second family is now contained in the Euclidean plane and so we can use Propostion \ref{P1}.}
\end{figure}

%%%%%%%%%%%%%%%%%%%%%%%%%%%%%%%%%%%%%%%%%%%%%%%%%%%%%%%%%%%%%%%%%%%%%%%%%%%%%%%%%%%%
%%%%%%%%%%%%%%%%%%%%%%%%%%%%%%%%%%%%%%%%%%%%%%%%%%%%%%%%%%%%%%%%%%%%%%%%%%%%%%%%%%%%
%%%%%%%%%%%%%%%%%%%%%%%%%%%%%%%%%%%%%%%%%%%%%%%%%%%%%%%%%%%%%%%%%%%%%%%%%%%%%%%%%%%%
%%%%%%%%%%%%%%%%%%%%%%%%%%%%%%%%%%%%%%%%%%%%%%%%%%%%%%%%%%%%%%%%%%%%%%%%%%%%%%%%%%%%

With Propositions \ref{P1} and \ref{C8P2} at hands, we can conclude the proof of Theorem \ref{T:MAIN}. Fix a large constant $C \gg 1$ and let $\epsilon \ll 1$ as in Proposition \ref{P1}. Because we suppose that $\Bar{A} \cap S_3 \neq \emptyset$, we can choose an set of angles $\boldsymbol{\alpha}_0 \in A$ such that if $\boldsymbol{\alpha}_1$ is the set of Euclidean angles as described in Proposition \ref{C8P2}, then $$p(\boldsymbol{\alpha}_1) < \epsilon.$$ In this situation, the local family $\mathcal{L}_{r'}(\varphi(T))$ given by Proposition \ref{C8P2} satisfies the condition of Proposition \ref{P1}: hence there exists a Kakeya-type set $F$ satisfying $$ \left| \left\{ M_{\mathcal{L}_{r'}(\varphi(T))}\mathbb{1}_F \geq \frac{1}{4} \right\} \right| \geq C\left|F \right|.$$ Define now the set $E$ in $\mathbb{H}^2$ as the pre-image of $F$ by the exponential map $$ E := \exp_{x_0}^{-1}(F)$$ and let us see why we can exploit the set $E$ with the maximal operator $M_\mathcal{B}$. Fix $\varphi(T')$ in $\mathcal{L}_{r'}(\varphi(T))$ such that $$\left| \varphi(T') \cap F \right| \geq \frac{1}{4}\left|\varphi(T') \right|$$ and observe that in $\mathbb{H}^2$ we have $$\mu\left( T' \cap E \right) \simeq \left| \exp_{x_0}(T' \cap E)\right|  \geq \left| \varphi(T') \cap F\right|  \geq \frac{1}{4}\left| \varphi(T') \right| \simeq \frac{1}{40}\mu\left( T' \right).$$ Finally, the following estimate follows by inclusion $$ \mu\left(\left\{ M_\mathcal{B}\mathbb{1}_E \geq \frac{1}{40}\right\} \right)\geq  \mu\left( \bigcup_{T'} T' \right) \simeq \left| \bigcup_{T'} \varphi(T') \right| \geq C\left| E \right| \simeq C\mu\left(F \right).$$ This concludes the last part of the proof of Theorem \ref{T:MAIN}

{}

\end{document}